\documentclass[12pt]{article}
\usepackage[a4paper,margin=1in]{geometry}
\usepackage{amsmath,amssymb,amsthm,mathtools}
\usepackage{bm}
\usepackage{cite}
\usepackage{enumitem}
\usepackage{microtype}
\usepackage{mathrsfs}
\usepackage{setspace}
\usepackage[hidelinks]{hyperref}
\hypersetup{pdftitle={Type-II Singularities of Immersed Locally Convex Curves Evolving under Generalized Anisotropic Gage Area-Preserving Flows with Slow Curvature Speeds},pdfauthor={Ze-Yu Ye and Xiao-Liu Wang}}
\allowdisplaybreaks
\numberwithin{equation}{section}

\newtheorem{theorem}{Theorem}[section]
\newtheorem{lemma}[theorem]{Lemma}
\newtheorem{proposition}[theorem]{Proposition}
\newtheorem{corollary}[theorem]{Corollary}
\newtheorem{remark}[theorem]{Remark}

\newcommand{\R}{\mathbb{R}}
\newcommand{\Sm}{\mathbb{S}_{m}}
\newcommand{\ks}{\kappa_{\sigma}}
\newcommand{\Ls}{L_{\sigma}}
\newcommand{\Ws}{W_{\sigma}}
\newcommand{\ph}{\varphi}

\begin{document}

\renewcommand{\thefootnote}{\fnsymbol{footnote}}
\begin{center}
{\LARGE Type-II Singularities of Immersed Locally\\[2mm]
Convex Curves Evolving under Generalized\\[2mm]
Anisotropic Gage Area-Preserving Flows\\[2mm]
with Slow Curvature Speeds$^{*}$\par}
\vspace{0.9cm}
{\large Ze-Yu Ye$^{1}$,\qquad Xiao-Liu Wang$^{2\dagger}$\par}
\vspace{0.25cm}
{\itshape 1. School of Mathematics, Southeast University, Nanjing 210096, Jiangsu, China.\par}
{\itshape Email: 230268271@seu.edu.cn\par}
\vspace{0.18cm}
{\itshape 2. School of Mathematics, Southeast University, Nanjing 210096, Jiangsu, China.\par}
{\itshape Email: xlwang@seu.edu.cn\par}
\vspace{0.55cm}
October 7, 2026
\end{center}

\footnotetext[2]{Corresponding author: Xiao-Liu Wang}
\renewcommand{\thefootnote}{\arabic{footnote}}

\begin{abstract}
In this paper, we study the singularity behavior of immersed locally convex closed plane curves evolving under generalized anisotropic Gage area-preserving flows with slow curvature speeds of the form $\kappa_\sigma^\alpha$, where $0<\alpha\le1$. Here $\kappa_\sigma$ denotes the anisotropic curvature associated with a prescribed anisotropy function $\sigma$. We prove that every finite-time blow-up is of Type-II provided that the Euclidean length of the evolving curve remains uniformly bounded away from zero. This result holds without any symmetry assumptions on either the initial curve or the anisotropy. When the initial signed area is nonzero, this length lower bound is automatically satisfied, and thus any finite-time singularity is necessarily of Type-II. By means of Hamilton's rescaling method combined with a differential Harnack estimate, we show that along a suitable blow-up subsequence, the rescaled curves at the selected times converge locally to the profile of a translating solution of the corresponding anisotropic curvature flow \(V=\kappa_\sigma^\alpha\). We also establish sufficient conditions for finite-time singularity formation in terms of the initial signed area and anisotropic length. Moreover, we construct locally convex initial curves arbitrarily close to the Wulff curve for which a finite-time Type-II singularity occurs.
\end{abstract}


\section{Introduction}
\subsection{Background and notation}

A plane curvature flow is a one-parameter family of smooth closed curves
\[
 X(\cdot,t):\Sm:=\R/(2m\pi\mathbb Z)\longrightarrow\R^2,
 \qquad t\in[0,T),
\]
which evolves in the direction of its inner unit normal $N$ with speed $V$:
\begin{equation}\label{eq:general-flow}
 \frac{\partial X}{\partial t}(u,t)=V(u,t,\kappa,\theta)N.
\end{equation}
Here $T=X_s$ is the unit tangent, $s$ is arc length, $V$ may depend on the signed curvature $\kappa$, position parameter $u$, tangential angle $\theta$ and time parameter $t$, and the orientation is chosen so that
\[
 N=R_{\pi/2}T,\qquad T_s=\kappa N,
\]
where $R_{\pi/2}$
denotes the counterclockwise rotation matrix in $\R^2$
through an angle $\pi/2$.
The rotation index of a smooth closed curve $X$ is
\[
 m=\frac{1}{2\pi}\int_{X}\kappa\,ds.
\]
Suppose $X$ has strictly positive curvature everywhere. We say that $X$ is \textbf{convex} if \(m=1\), and \textbf{locally convex} if \(m\ge 2\). A locally convex curve with $m\ge2$ is in general immersed and may have self-intersections. Its Euclidean length and signed area are
\begin{equation}\label{eq:length-area}
 L=\int_{X}ds,
 \qquad
 A=-\frac12\int_{X}\langle X,N\rangle\,ds.
\end{equation}
For an immersed curve, $A$ is the algebraic area and may be positive, zero, or negative. Since $\kappa>0$, the tangential angle $\theta$ is a global parameter on $\Sm$ and $\kappa\,ds=d\theta$.
 We now recall some terminology related to the blow-up phenomenon for the power-type curvature flow. For $\alpha>0$, a power-type curvature flow generally takes the form of
 \begin{equation}\label{eq:3}
\frac{\partial X}{\partial t}= \big(\Psi(\theta)\kappa^\alpha-\lambda(t)\big)\,N,\qquad t\in[0,T), 
\end{equation}
where $\Psi(\theta)>0$ is an anisotropic function and $\lambda(t)$ is a nonlocal term related to some geometric quantities such as length and area. We call the flow anisotropic if $\Psi$ is not a constant and isotropic otherwise. We say the flow \eqref{eq:3} develops a singularity at time $T\,(0<T\le\infty)$ if the curvature $\kappa(\cdot,t)$ of the evolving curves $X_t$ satisfies
\[
 \limsup_{t\uparrow T}\kappa_{max}(t)=\infty,
\]
where $\kappa_{max}(t):=\max_{u\in\Sm}\kappa(u,t)$.
 For $\alpha>0$ and $T<\infty$, the singularity is called \textbf{Type-I} if
\begin{equation}\label{eq:type-I-definition}
 \limsup_{t\uparrow T}(T-t)\kappa_{max}^{\alpha+1}(t)<\infty,
\end{equation}
and \textbf{Type-II} if
\begin{equation}\label{eq:type-II-definition}
 \limsup_{t\uparrow T}(T-t)\kappa_{max}^{\alpha+1}(t)=\infty.
\end{equation}

The classical curve shortening flow (CSF) is given by
\begin{equation}\label{eq:CSF}
 V=\kappa.
\end{equation}
Gage--Hamilton \cite{GageHamilton1986} proved that every convex embedded closed curve evolving by CSF shrinks to a point in finite time, with a Type-I singularity.  Then Grayson \cite{Grayson1987} extended the round-point theorem to the case of arbitrary embedded closed curves.  For immersed curves the singularity theory is much richer. Angenent \cite{Angenent1991} studied singularity formation for the CSF and showed that, in the locally convex setting, both Type-I and Type-II singularities can occur. Type-I blow-up is related to homothetically shrinking solutions, the Abresch--Langer curves, whereas Type-II rescaling leads to a translating solution, the Grim Reaper (see \cite{AbreschLanger1986,Angenent1991}).

A natural nonlinear extension is the power-type curvature flow
\begin{equation}\label{eq:power-local}
 V=\frac1\alpha\kappa^\alpha,\qquad \alpha\neq 0,
\end{equation}
or, in the anisotropic setting,
\begin{equation}\label{eq:anis-local}
 V=\frac1\alpha\Psi(\theta)\kappa^\alpha,\qquad
 \alpha\neq 0,\qquad \Psi(\theta)>0.
\end{equation}
For the anisotropic flow \eqref{eq:anis-local} with embedded
convex initial curves, Andrews \cite{Andrews1998} developed
a general theory describing how the exponent $\alpha$ affects
the evolution. In particular, only a Type-I singularity occurs
when $\alpha>1/3$. Andrews \cite{Andrews2002} further
studied the instability of the flow \eqref{eq:anis-local}   for
$0<\alpha<1/3$ and showed that the isoperimetric ratio
may become unbounded, leading to a Type-II singularity. Such behavior may also occur at
the critical exponent $\alpha=1/3$ for suitable
anisotropies, whereas the isotropic critical case
has only a Type-I singularity. See also Poon--Tsai
\cite{PoonTsai2015} for the relation between the
isoperimetric ratio and the blow-up type.
In the isotropic setting $\Psi\equiv1$, Andrews \cite{Andrews2003}
classified the closed homothetically shrinking
solutions of \eqref{eq:power-local}.  Other related curvature flows with forcing terms
were studied by Chou--Zhu \cite{ChouZhu1999}, who considered
the anisotropic flow
$\beta(\theta)V=g(\theta)\kappa+F$,
where $\beta,g>0$ are smooth $2\pi$-periodic functions and $F$ is a constant.
In the shrinking case, for embedded convex initial curves,
they proved finite-time extinction with Type-I blow-up.
Moreover, the rescaled curves converge smoothly along
a subsequence to a homothetically shrinking solution
of the unforced flow $\beta(\theta)V=g(\theta)\kappa$.
For immersed locally convex curves moving by \eqref{eq:power-local}, Poon–Tsai \cite{PoonTsai2010} treated the fast
regime $\alpha>1$ and proved that finite-time blow-up is of Type-I. Lin--Poon--Tsai \cite{LinPoonTsai2012} studied the slow regime $0<\alpha\le1$ and analyzed both Type-I and Type-II singularities. Moreover, Type-I blow-up is associated with homothetic solutions, while under a special symmetry assumption, a Type-II blow-up develops a translator profile. For the anisotropic flow \eqref{eq:anis-local}, Poon--Tsai  \cite{PoonTsai2015} investigated the corresponding degenerate parabolic equation.  Under the Type-I assumption, they established  subsequential convergence to self-similar profiles, and under the Type-II assumption, they showed that, in the case where the equation defines an embedded convex flow, the blow-up profile near a maximum point is a translating solution of \eqref{eq:anis-local}. For the anisotropic flow \eqref{eq:anis-local} with $\alpha>1$, Wang–Wang \cite{WangWang2023} proved
that for immersed locally convex closed initial curves, any finite-time singularity is of Type-I. They also established subsequential convergence of the rescaled curvature and further obtained full-time convergence under suitable symmetry assumptions on both the anisotropy and the initial curve.
\vskip 10 pt 
We next turn to nonlocal curvature flows. Gage \cite{Gage1986} introduced the area-preserving curvature flow
\begin{equation}\label{eq:Gage-isotropic}
 V=\kappa-\frac{2\pi}{L(t)}
\end{equation}
for embedded convex curves and established global existence together with smooth convergence to a circle. A power-type generalization
\begin{equation}\label{eq:power-AP-isotropic}
 V=\kappa^\alpha-\frac1{L(t)}\int_{X_t}\kappa^\alpha\,ds,
 \qquad \alpha>0,
\end{equation}
was studied for embedded convex curves by Tsai--Wang \cite{TsaiWang2015}.  Other important isotropic nonlocal flows include the length-preserving flow due to Ma--Zhu \cite{MaZhu2012} and the gradient flow of the isoperimetric ratio proposed by Jiang--Pan \cite{JiangPan2008}. For embedded convex initial curves,  these flows typically possess global solutions and converge smoothly to a circle. 
\vskip 10 pt 
For immersed curves, finite-time singularities can occur under suitable initial conditions. For the immersed area-preserving curvature flow
\begin{equation}\label{eq:Gage-immersed}
 V=\kappa-\frac{2m\pi}{L(t)},
\end{equation}
Escher--Ito \cite{EscherIto2005} proved finite-time singularity formation when $A(0)<0$, or when $m\ge2$ and $L^2(0)<4m\pi A(0)$. Nagasawa--Nakamura \cite{NagasawaNakamura2021} studied several nonlocal curvature flows for plane curves with general rotation number, including the area-preserving flow \eqref{eq:Gage-immersed}. They obtained convergence results for global solutions, together with finite-time blow-up criteria, upper estimates for the blow-up time, and lower estimates for the blow-up rate. For the power-type area-preserving flow \eqref{eq:power-AP-isotropic}, \v{S}e\v{s}um--Tsai--Wang \cite{SesumTsaiWang2020} studied immersed locally convex curves and proved finite-time singularity formation under the explicit conditions
\begin{equation}\label{eq:STW-condition}
 A(0)<0
 \qquad\hbox{or}\qquad
 L^2(0)<4m\pi A(0).
\end{equation}
They also obtained convergence to an $m$-fold circle for suitable global solutions. Lippoth \cite{Lippoth2025} obtained further blow-up results for general immersed closed curves under \eqref{eq:Gage-immersed}, including a sufficient condition for finite-time curvature blow-up when the initial signed area is positive and the rotation index is one. He also proved curvature blow-up for initial curves with zero signed area.  Furthermore, people are also interested in the \textbf{type} and the \textbf{profile} of singularities. In the isotropic case, M\"ader-Baumdicker \cite{Mader2018} studied Gage's area-preserving curvature flow with Neumann free boundary conditions and, under suitable assumptions, proved that any finite-time singularity must be of Type-II. Moreover, for convex initial curves, she employed Hamilton's rescaling technique to show that the asymptotic profile is a Grim Reaper or a half Grim Reaper. Dittberner \cite{Dittberner2021} investigated isotropic curvature flows with global forcing terms, including Gage's area-preserving curvature flow, for embedded curves satisfying suitable geometric conditions. She proved global existence for a class of embedded curves by first ruling out Type-I singularities and then excluding Type-II singularities using a distance comparison principle. More recently, Anada--Ishiwata--Ushijima \cite{AnadaIshiwataUshijima2025} considered the isotropic power-type area-preserving flow \eqref{eq:power-AP-isotropic} for $0<\alpha\le1$ with a special class of highly symmetric Abresch--Langer type immersed initial curves, and proved Type-II blow-up. 
\vskip 10 pt 
A natural problem is what happens if we consider nonlocal curvature flows in the anisotropic setting. Let $\sigma\in C^\infty(\R)$ be a positive $2m\pi$-periodic function satisfying
\begin{equation}\label{eq:sigma}
 \sigma(\theta)>0,
 \qquad
 \ph(\theta):=\sigma(\theta)+\sigma''(\theta)>0,
 \qquad \theta\in\Sm.
\end{equation}
Motivated by Green--Osher \cite{GreenOsher1999} and \v{S}ev\v{c}ovi\v{c}--Yazaki \cite{SevcovicYazaki2013}, we define the anisotropic curvature and anisotropic length by
\begin{equation}\label{eq:anis-def}
 \ks=\ph(\theta)\kappa,
 \qquad
 \Ls=\int_X\sigma(\theta)\,ds.
\end{equation}
The associated Wulff curve is parametrized by
\begin{equation}\label{eq:wulff-param}
 \partial\Ws
 =-\sigma(\theta)N(\theta)+\sigma'(\theta)T(\theta),\quad
 \theta\in\Sm.
\end{equation}
This parametrization defines a closed, convex curve when $m=1$, and a closed locally convex curve when $m\ge2$, which may have self-intersections or be multiply covered. Its Euclidean curvature is $\ph^{-1}$ and its anisotropic curvature is identically one. We denote its signed area by $|\Ws|$. Direct calculation gives
\begin{equation}\label{eq:wulff-area}
 |\Ws|
 =\frac12\int_{\Sm}\sigma(\theta)\ph(\theta)\,d\theta
 =\frac12\Ls(\partial\Ws)>0.
\end{equation}
The flow $V=\ks$ is the gradient flow of the anisotropic interfacial energy $\int_X\sigma\,ds$; see Taylor--Cahn \cite{TaylorCahn1994}.  If we consider the gradient flow of the surface energy with the area-preserving constraint, then the following anisotropic version of Gage's area-preserving curvature flow is obtained:
\begin{equation}\label{eq:anis-Gage}
 \frac{\partial X}{\partial t}
 =\left(
 \ks-\frac1{L(t)}\int_X\ks\,ds
 \right)N.
\end{equation}
Wang--Wang \cite{WangWang2025} proved that for embedded smooth convex closed initial curves, the flow \eqref{eq:anis-Gage} exists globally and converges smoothly, up to translation, to the boundary of the homothetic Wulff shape $\sqrt{{A(0)}/{|W_\sigma|}}W_\sigma$. Liu--Tsai--Wang \cite{LiuTsaiWang2025} considered a power-type extension
\begin{equation}\label{eq:anis-power-AP}
 \frac{\partial X}{\partial t}
 =\left(
 \ks^\alpha-\frac1{L(t)}\int_X\ks^\alpha\,ds
 \right)N,\qquad \alpha>0,
\end{equation}
and established analogous global existence and smooth convergence results for embedded convex closed initial curves. There are several other related anisotropic nonlocal or constrained curve flows. We refer, for example, to Bene\v{s}--Yazaki--Kimura \cite{BenesYazakiKimura2011}, Pan--Yang \cite{PanYang2019}, Gao--Zhang \cite{GaoZhang2019}, Li--Wang \cite{LiWang2023}, Liu--Yang--Zhang \cite{LiuYangZhang2024}, Zhao--Yang--Mao--Fang \cite{ZhaoYangMaoFang2024}, Sun \cite{Sun2025}, and Ye--Wang \cite{YeWang2026}. See also the survey of Giga--Po\v{z}\'ar \cite{GigaPozar2022}. 

The present paper focuses on generalized anisotropic Gage area-preserving flows with slow curvature speeds, corresponding to the regime $0<\alpha\le1$ in \eqref{eq:anis-power-AP}:
\begin{equation}\label{eq:main-flow}
 \frac{\partial X}{\partial t}
 =\left(\ks^\alpha-\lambda(t)\right)N,
 \qquad
 \lambda(t)=\frac1{L(t)}\int_{X_t}\ks^\alpha\,ds,
 \qquad 0<\alpha\le1.
\end{equation}
Since
\begin{equation}\label{eq:total-anis}
 \int_{X_t}\ks\,ds
 =\int_{\Sm}\ph(\theta)\,d\theta
 =:K_\sigma,
\end{equation}
Jensen's inequality gives
\begin{equation}\label{eq:lambda-upper}
 0<\lambda(t)\le\left(\frac{K_\sigma}{L(t)}\right)^\alpha.
\end{equation}
The flow preserves the signed area and decreases the anisotropic length. We first show that every finite-time singularity with $\inf_{t<T}L(t)>0$ is of Type-II, and we determine its local blow-up profile by applying Hamilton's rescaling method. We then provide sufficient conditions on immersed locally convex initial curves that guarantee finite-time singularity formation. In contrast to Anada--Ishiwata--Ushijima \cite{AnadaIshiwataUshijima2025}, we impose no symmetry condition, and the anisotropy may be general under \eqref{eq:sigma}.

\subsection{Main results}

Our first two theorems concern the type and local profile of a finite-time singularity under a positive length lower bound. 

\begin{theorem}[Type-II blow-up]\label{thm:typeII}
Let $0<\alpha\le1$, and let $X(\cdot,t),
 \, 0\le t<T<\infty,$ be a smooth immersed locally convex solution of \eqref{eq:main-flow}. Assume
\[
 \limsup_{t\uparrow T}\kappa_{max}(t)=\infty,
 \qquad
 \inf_{0\le t<T}L(t)>0.
\]
Then the singularity at $T$ is Type-II, namely
\begin{equation}\label{eq:typeII-main}
 \limsup_{t\uparrow T}(T-t)\kappa_{max}^{\alpha+1}(t)=\infty.
\end{equation}
\end{theorem}

The length lower bound in Theorem~\ref{thm:typeII} cannot be omitted in general, already in the isotropic case $\alpha=1$. In a recent preprint, Cernomazov \cite[Main Result~A, Remarks~1.7 and~2.1]{Cernomazov2026} constructed noncircular locally convex shrinkers of the isotropic flow \eqref{eq:Gage-immersed}. These have zero signed area and shrink by a factor proportional to $\sqrt{T-t}$, so $L(t)\to0$ and the singularity is of Type-I.

Since the signed area is preserved and an immersed closed plane curve satisfies $L^2\ge4\pi|A|$, the lower bound for $L(t)$ is automatic whenever $A(0)\ne0$.
\vskip 10 pt 
Our next theorem identifies the curvature profile near a Type-II blow-up point.

\begin{theorem}[Blow-up profile]\label{thm:profile}
Let $0<\alpha\le1$, and let $X(\cdot,t)$, $0\le t<T<\infty$, be a smooth immersed locally convex solution of \eqref{eq:main-flow} with a Type-II singularity at $T$, and assume
\begin{equation}\label{eq:profile-length}
 \inf_{0\le t<T}L(t)>0.
\end{equation}
Then there exist times $t_j\uparrow T$ and angles $\theta_j\in\Sm$ at which $\ks(\cdot,t_j)$ attains its maximum such that, after passing to a subsequence,
\[
 \theta_j\longrightarrow\theta_*\in\Sm
\]
and
\begin{equation}\label{eq:profile-moving}
 \frac{\ks(\theta_j+\xi,t_j)}{\ks(\theta_j,t_j)}
 \longrightarrow (\cos\xi)^{1/\alpha}
\end{equation}
uniformly for $\xi\in[-\pi/2,\pi/2]$. Equivalently,
\begin{equation}\label{eq:profile-fixed}
 \frac{\kappa(\theta_*+\xi,t_j)}{\kappa(\theta_*,t_j)}
 \longrightarrow
 \frac{\ph(\theta_*)}{\ph(\theta_*+\xi)}(\cos\xi)^{1/\alpha}
\end{equation}
uniformly on $[-\pi/2,\pi/2]$.
\end{theorem}

The limiting curvature profile in Theorem~\ref{thm:profile} is realized, after the following rescaling, by a translating solution of the corresponding anisotropic flow $V=\ks^\alpha$. Put $q_j=\ks^\alpha(\theta_j,t_j)$, $Q_j=q_j^{1/\alpha}$ and $p=1+1/\alpha$, and let $\tau=q_j^p(t-t_j)$ denote the rescaled time. Define
\[
 \widetilde X_j(\xi,\tau)=Q_j\left[
 X\left(\theta_j+\xi,t_j+\frac{\tau}{q_j^p}\right)-X(\theta_j,t_j)
 \right].
\]
Spatial dilation by $Q_j$ divides curvature by $Q_j$ and preserves the tangent direction. Thus \eqref{eq:profile-moving} is precisely the convergence of the anisotropic curvature of $\widetilde X_j(\cdot,0)$ to $(\cos\xi)^{1/\alpha}$. Moreover,
\[
 \partial_\xi\widetilde X_j(\xi,0)
 =\frac{\ph(\theta_j+\xi)}{\bigl(\ks^\alpha(\theta_j+\xi,t_j)/q_j\bigr)^{1/\alpha}}
 T(\theta_j+\xi).
\]
On compact subintervals of $(-\pi/2,\pi/2)$ the denominator converges to a positive function. Since $\widetilde X_j(0,0)=0$, integration gives $C^1_{\mathrm{loc}}$ convergence to
\begin{equation}\label{eq:translator-param}
 X_\infty(\xi,0)=
 \int_0^\xi \frac{\ph(\theta_*+r)}{(\cos r)^{1/\alpha}}T(\theta_*+r)\,dr,
 \qquad |\xi|<\frac\pi2.
\end{equation}
The positive factor in $\partial_\xi X_\infty$ shows that its unit tangent and normal are $T_\infty=T(\theta_*+\xi)$ and $N_\infty=N(\theta_*+\xi)$. Using $ds_\infty/d\xi=\ph(\theta_*+\xi)/(\cos\xi)^{1/\alpha}$ and $\partial_\xi T_\infty=N_\infty$, we obtain
\[
 \kappa_\infty(\xi)=\frac{(\cos\xi)^{1/\alpha}}{\ph(\theta_*+\xi)},
 \qquad \kappa_{\sigma,\infty}(\xi)=(\cos\xi)^{1/\alpha}.
\]
Set $e=N(\theta_*)$. Since $e\cdot N_\infty=\cos\xi$, the family
\begin{equation}\label{eq:translator-motion}
 X_\infty(\xi,\tau)=X_\infty(\xi,0)+\tau e
\end{equation}
satisfies the normal-velocity law
\begin{equation}\label{eq:local-anis-trans}
 \partial_\tau X_\infty\cdot N_\infty=\kappa_{\sigma,\infty}^\alpha.
\end{equation}
Indeed, its full velocity in the fixed-angle parametrization is
$\kappa_{\sigma,\infty}^\alpha N_\infty-(\kappa_{\sigma,\infty}^\alpha)_\xi T_\infty
=\cos\xi\,N_\infty+\sin\xi\,T_\infty=e$.
Thus it is a translating solution of the anisotropic curvature flow $V=\ks^\alpha$, up to tangential reparametrization. Since $0<\alpha\le1$, its arc length diverges at both endpoints of the angle interval. When $\alpha=1$ and $\sigma\equiv1$, a rigid motion gives
$X_\infty(\xi,\tau)=(\xi,-\log\cos\xi)+\tau(0,1)$, the classical Grim Reaper.
\vskip 10 pt 
We now give sufficient conditions on the initial curve that guarantee finite-time singularity formation.
\begin{theorem}[Finite-time singularity]\label{thm:finite}
Let $0<\alpha\le1$, and let $X_0$ be a smooth immersed locally convex closed initial curve with rotation index $m\ge2$. Suppose that either
\begin{equation}\label{eq:finite-cond}
 A(0)<0
 \qquad\hbox{or}\qquad
 0<\Ls^2(0)<4|\Ws|A(0).
\end{equation}
Then the solution of \eqref{eq:main-flow} develops a singularity at a finite maximal time $T<\infty$.
\end{theorem}

\begin{remark}[Initial curves near the Wulff curve]\label{rem:wulff-perturbation}
For every $\sigma$ satisfying \eqref{eq:sigma} and every $m\ge2$, the second condition in \eqref{eq:finite-cond} is satisfied by locally convex closed initial curves arbitrarily close to $\partial\Ws$ in $C^\infty(\Sm)$. The strict inequality persists under sufficiently small smooth perturbations, so these examples need not be symmetric. The construction is given in Subsection~\ref{subsec:wulff-examples}.
\end{remark}

Both conditions in \eqref{eq:finite-cond} imply $A(0)\ne0$. Therefore Theorems~\ref{thm:finite} and \ref{thm:typeII} immediately give the following result.

\begin{corollary}\label{cor:finite-typeII}
If $X_0$ satisfies either condition in \eqref{eq:finite-cond}, then the solution of \eqref{eq:main-flow} develops a finite-time Type-II singularity.
\end{corollary}

The paper is organized as follows. Section~\ref{sec:prelim} gives the evolution equations, the monotonicity of geometric quantities, the gradient estimate and a concentration estimate near points of maximal curvature. Section~\ref{sec:typeII} proves Theorem~\ref{thm:typeII} by showing that the Type-I assumption would force the Euclidean length to vanish. Section~\ref{sec:profile} applies Hamilton's rescaling method and proves the compactness, nonlocal-error, and Harnack estimates needed to identify the blow-up profile in Theorem~\ref{thm:profile}. Finally, Section~\ref{sec:global} establishes uniform curvature bounds for global solutions whose anisotropic length has a positive limit, identifies the limiting Wulff curve, and proves the finite-time singularity criteria in Theorem~\ref{thm:finite}. It also gives the construction described in Remark~\ref{rem:wulff-perturbation}.

\section{Preliminaries}\label{sec:prelim}

\subsection{Evolution equations}

For a locally convex solution of \eqref{eq:main-flow}, one can use the tangential angle to parameterize it. Put $v(\theta,t):=\ks^\alpha(\theta,t)$ and $p=1+1/\alpha\ge2$. A tangential velocity changes only the parametrization, not the geometric image of the curve; see Chou--Zhu \cite{ChouZhu2001} and Gage \cite{Gage1986}. If $V$ and $G$ are the normal and tangential velocities, then $\partial_t\theta=\partial_sV+\kappa G$. Choosing $G=-v_\theta$ with $V=v-\lambda(t)$ therefore keeps $\theta$ independent of time and gives
\begin{equation}\label{eq:X-fixed}
 \frac{\partial X}{\partial t}=(v-\lambda)N-v_\theta T.
\end{equation}
At fixed angle the curvature satisfies $\kappa_t=\kappa^2(V_{\theta\theta}+V)$. Since $\ph$ is independent of time and $v=(\ph\kappa)^\alpha$, this gives
\begin{equation}\label{eq:v-main}
 v_t=\frac{\alpha}{\ph(\theta)}v^p
 \left(v_{\theta\theta}+v-\lambda(t)\right)
 \qquad\hbox{on }\Sm\times(0,T),
\end{equation}
where
\begin{equation}\label{eq:lambda-main}
 \lambda(t)=\frac{J(t)}{L(t)},
 \qquad J(t):=\int_{\Sm}\ph v^{1-1/\alpha}\,d\theta,
 \qquad K_\sigma:=\int_{\Sm}\ph(\theta)\,d\theta.
\end{equation}
We also have
\begin{equation}\label{eq:ds-main}
 ds=\frac{\ph}{v^{1/\alpha}}\,d\theta,
 \qquad L(t)=\int_{\Sm}\frac{\ph}{v^{1/\alpha}}\,d\theta,
 \qquad \Ls(t)=\int_{\Sm}\frac{\sigma\ph}{v^{1/\alpha}}\,d\theta.
\end{equation}
In particular, the Euclidean length satisfies
\begin{equation}\label{eq:Lprime-basic}
 L'(t)=-\int_{\Sm}(v-\lambda)\,d\theta
       =-\int_{\Sm}v\,d\theta+2m\pi\lambda.
\end{equation}
This follows either by differentiating \eqref{eq:ds-main} and using periodicity, or from the first variation of length. Notice that $\lambda$ is the arc-length average of $v$, so $0<\lambda\le v_{max}$.

Short-time existence and preservation of local convexity follow by the standard quasilinear parabolic argument used for anisotropic nonlocal curvature flows; see \cite{ChouZhu2001,LiuTsaiWang2025}. In particular, the solution can be continued beyond any finite time $T$, provided its curvature remains uniformly bounded on $[0,T)$. Indeed, if $v\le M$, then $\lambda\le M$ and the spatial minimum satisfies
\[
 \frac{d}{dt}v_{\min}\ge-\frac{\alpha M}{\ph_{\min}}v_{\min}^p
\]
in the lower Dini derivative sense. Thus $v$ stays uniformly positive on every bounded time interval, and the equation remains uniformly parabolic when the curvature is bounded.

\subsection{Preserved area and non-increasing anisotropic length}

\begin{lemma}\label{lem:monotone}
Along \eqref{eq:main-flow}, the signed area is preserved and the anisotropic length is non-increasing, i.e.,
\begin{equation}\label{eq:Ls-dissipation}
 \frac{dA}{dt}=0,
 \qquad
 \frac{d\Ls}{dt}
 =-\int_{X_t}(\ks-\lambda^{1/\alpha})(\ks^\alpha-\lambda)\,ds\le0.
\end{equation}
\end{lemma}

\begin{proof}
The standard variation formulas give
\[
 \frac{dA}{dt}=-\int_{X_t}(\ks^\alpha-\lambda)\,ds=0
\]
by the definition of $\lambda$. Also
\[
 \frac{d\Ls}{dt}
 =-\int_{X_t}\ks(\ks^\alpha-\lambda)\,ds
 =-\int_{\Sm}\ph v\,d\theta+\lambda K_\sigma.
\]
Since $\int_{X_t}(\ks^\alpha-\lambda)\,ds=0$,
\[
 \int_{X_t}\ks(\ks^\alpha-\lambda)\,ds
 =\int_{X_t}(\ks-\lambda^{1/\alpha})(\ks^\alpha-\lambda)\,ds.
\]
The integrand is nonnegative because $s\mapsto s^\alpha$ is increasing on $(0,\infty)$, which proves \eqref{eq:Ls-dissipation}.
\end{proof}

\subsection{A gradient estimate}

The following gradient estimate will be used in both the global argument and the blow-up analysis.

\begin{lemma}[Gradient estimate]\label{lem:gradient}
Let $v>0$ solve \eqref{eq:v-main}. Set
\[
 \Psi=v^2+v_\theta^2.
\]
Then for every $t<T$,
\begin{equation}\label{eq:gradient-est}
 \max_{\Sm\times[0,t]}\Psi
 \le
 \max\left\{
 C_\ph\max_{\Sm\times[0,t]}v^2,
 \max_{\Sm}(v^2+v_\theta^2)(\theta,0)
 \right\},
\end{equation}
where
\[
 C_\ph=1+\frac{\|\ph_\theta\|_\infty^2}{4\ph_{min}^2}.
\]
\end{lemma}

\begin{proof}
Fix $t>0$ and suppose the maximum of $\Psi$ on $\Sm\times[0,t]$ is attained at $(\theta_*,t_*)$ with $t_*>0$. At this point
\[
 \Psi_\theta=2v_\theta(v+v_{\theta\theta})=0,
 \qquad
 \Psi_{\theta\theta}\le0.
\]
If $v_\theta(\theta_*,t_*)=0$, then $\Psi=v^2$ and there is nothing to prove. Otherwise
\[
 v_{\theta\theta}+v=0
\]
at $(\theta_*,t_*)$. Differentiating \eqref{eq:v-main} and evaluating at this point gives
\begin{align*}
 0\le \Psi_t
 &=-\frac{2\alpha\lambda v^{p+1}}{\ph}
   -\frac{2\alpha p\lambda v^{p-1}v_\theta^2}{\ph}
   +\frac{2\alpha\lambda\ph_\theta v^pv_\theta}{\ph^2}
   +\frac{\alpha v^p}{\ph}\Psi_{\theta\theta} \\
 &\le
 -\frac{2\alpha\lambda v^{p-1}}{\ph}
 \left(2v_\theta^2-\frac{\ph_\theta}{\ph}vv_\theta\right),
\end{align*}
where $p\ge2$ has been used and the non-positive terms have been discarded in the last line. Hence
\[
 |v_\theta|
 \le \frac{|\ph_\theta|}{2\ph}\,v
 \le \frac{\|\ph_\theta\|_\infty}{2\ph_{min}}v.
\]
Thus $\Psi\le C_\ph v^2$ at every interior spacetime maximum with $v_\theta\ne0$, and \eqref{eq:gradient-est} follows.
\end{proof}

\begin{lemma}[Concentration estimate]\label{lem:concentration}
Suppose
\[
 v(\theta_*,t_*)=\max_{\Sm\times[0,t_*]}v.
\]
For every $0<\varepsilon<1/2$, there exist $\delta>0$ and $C_0>0$, independent of $t_*$, such that
\begin{equation}\label{eq:concentration}
 (1-\varepsilon)v(\theta_*,t_*)
 \le v(\theta,t_*)+C_0\varepsilon
\end{equation}
whenever  $|\theta-\theta_*|<\delta$.
\end{lemma}

\begin{proof}
By Lemma~\ref{lem:gradient},
\[
 |v_\theta(\theta,t_*)|
 \le\sqrt{C_I+C_\ph v^2(\theta_*,t_*)},
\]
where $C_I:=\max_{\Sm}(v^2+v_\theta^2)(\theta,0)$ depends only on the initial data. Integrating from $\theta$ to $\theta_*$ and choosing
$\delta=\varepsilon/\sqrt{C_\varphi}$ gives
\eqref{eq:concentration}.
\end{proof}

\section{Type-II singularities}\label{sec:typeII} 

This section proves Theorem~\ref{thm:typeII}. Assume for contradiction that the singularity is of Type-I. Set
\begin{equation}\label{eq:rU}
 r(t)=(T-t)^{1/(\alpha+1)},
 \qquad
 U(\theta,t)=r(t)^\alpha v(\theta,t).
\end{equation}
The Type-I assumption and Lemma~\ref{lem:gradient} imply
\begin{equation}\label{eq:U-bound}
 \|U(\cdot,t)\|_{L^\infty(\Sm)}
 +\|U_\theta(\cdot,t)\|_{L^\infty(\Sm)}\le C.
\end{equation}
At a spatial maximum of $v$, equation \eqref{eq:v-main} gives, in the upper Dini derivative sense,
\[
 \frac{d}{dt}v_{max}(t)
 \le \frac{\alpha}{\ph_{min}}v_{max}^{p+1}(t).
\]
Since $T$ is a curvature blow-up time, integration toward $T$ gives
\begin{equation}\label{eq:lower-blowup}
 v_{max}(t)\ge c(T-t)^{-1/p}=c(T-t)^{-\alpha/(\alpha+1)}
\end{equation}
for $t$ sufficiently close to $T$. Therefore
\begin{equation}\label{eq:U-positive-max}
 \max_{\Sm}U(\theta,t)\ge c_0>0.
\end{equation}

The assumption $\inf_{t<T}L(t)>0$ and \eqref{eq:lambda-upper} give a uniform bound for $\lambda$. From \eqref{eq:X-fixed} and \eqref{eq:U-bound},
\[
 |\partial_tX(\theta,t)|
 \le C(T-t)^{-\alpha/(\alpha+1)}.
\]
Hence there exists a continuous limiting map $X_T:\Sm\to\R^2$ satisfying
\begin{equation}\label{eq:X-limit-rate}
 \sup_{\theta\in\Sm}|X(\theta,t)-X_T(\theta)|
 \le C(T-t)^{1/(\alpha+1)}=Cr(t).
\end{equation}

\begin{lemma}\label{lem:U-lower}
There exist $c>0$ and $t_0<T$ such that
\[
 U(\theta,t)\ge c
\]
for all $\theta\in\Sm$ and $t\in[t_0,T)$.
\end{lemma}

\begin{proof}
Suppose not. Then there are $t_j\uparrow T$ such that
\[
 \min_{\Sm}U(\cdot,t_j)\longrightarrow0.
\]
By \eqref{eq:U-bound}, after passing to a subsequence,
\[
 U(\cdot,t_j)\longrightarrow U_\infty
\]
uniformly on $\Sm$, where $U_\infty\ge0$ is Lipschitz. By \eqref{eq:U-positive-max}, it is not identically zero, while it has at least one zero. Let $I=(a,b)$ be a connected component of $\{U_\infty>0\}$.

Put $r_j=(T-t_j)^{1/(\alpha+1)}$. Since
\[
 X_\theta(\theta,t_j)
 =r_j\ph(\theta)U(\theta,t_j)^{-1/\alpha}T(\theta),
\]
the length of $X(\cdot,t_j)$ restricted to any compact subinterval of $I$ tends to zero. Hence $X_T$ is constant on $I$. Write $X_T\equiv x_*$ there. Define
\[
 Y_j(\theta)=\frac{X(\theta,t_j)-x_*}{r_j}.
\]
By \eqref{eq:X-limit-rate}, $|Y_j|\le C$ on $I$, and
\[
 (Y_j)_\theta=\ph U(\theta,t_j)^{-1/\alpha}T.
\]
Since $T_\theta=N$,
\begin{equation}\label{eq:Yidentity}
 \frac{d}{d\theta}(Y_j\cdot T)
 =\ph U(\theta,t_j)^{-1/\alpha}+Y_j\cdot N.
\end{equation}
Integrating \eqref{eq:Yidentity} over $[a',b']\Subset I$ and using $|Y_j|\le C$, we have
\[
 \int_{a'}^{b'}\ph U(\theta,t_j)^{-1/\alpha}\,d\theta
 \le 2C+C(b'-a')\le C(2+2m\pi).
\]
Passing to the limit gives
\[
 \int_{a'}^{b'}\frac{\ph(\theta)}{U_\infty(\theta)^{1/\alpha}}\,d\theta\le C,
\]
where $C$ is independent of $a'$ and $b'$. Letting $a'\downarrow a$ and $b'\uparrow b$, we obtain
\begin{equation}\label{eq:Uintegrable}
 \int_I\frac{\ph(\theta)}{U_\infty(\theta)^{1/\alpha}}\,d\theta<\infty.
\end{equation}
But $U_\infty$ is Lipschitz and vanishes at an endpoint of $I$, so near that endpoint
\[
 U_\infty(\theta)\le C|\theta-a|.
\]
Since $1/\alpha\ge1$, the integral in \eqref{eq:Uintegrable} must diverge, a contradiction.
\end{proof}

\begin{proof}[Proof of Theorem~\ref{thm:typeII}]
By Lemma~\ref{lem:U-lower},
\[
 \kappa(\theta,t)=\frac{v^{1/\alpha}}{\ph}
 =\frac{U^{1/\alpha}}{r(t)\ph}
 \ge\frac{c}{(T-t)^{1/(\alpha+1)}}.
\]
Thus
\[
 L(t)=\int_{\Sm}\frac{d\theta}{\kappa(\theta,t)}
 \le C(T-t)^{1/(\alpha+1)}\longrightarrow0,
\]
contradicting $\inf_{t<T}L(t)>0$. Hence Type-I is impossible and the singularity is of Type-II.
\end{proof}

\section{Hamilton rescaling and the local profile}\label{sec:profile}

We now assume that a Type-II singularity occurs at $T<\infty$ and that \eqref{eq:profile-length} holds.

\subsection{Hamilton's rescaling method}

Following Hamilton's rescaling method \cite{Hamilton1995}, choose $s_j\uparrow T$ such that
\[
 (T-s_j)v_{max}^p(s_j)\longrightarrow\infty
\]
and set
\[
 T_j=\frac{T+s_j}{2}.
\]
Choose $(\theta_j,t_j)\in\Sm\times[0,T_j]$ so that
\begin{equation}\label{eq:Hj}
 H_j:=(T_j-t_j)v^p(\theta_j,t_j)
 =\max_{\theta\in\Sm,\,0\le t\le T_j}
 (T_j-t)v^p(\theta,t).
\end{equation}
Then
\[
 H_j\ge\frac12(T-s_j)v_{max}^p(s_j)\longrightarrow\infty.
\]
Since $v$ is bounded on compact time intervals in $[0,T)$, we have $t_j\to T$. After taking a subsequence, assume $t_j\uparrow T$ and put
\begin{equation}\label{eq:qj}
 q_j=v(\theta_j,t_j)\longrightarrow\infty,
 \qquad Q_j=q_j^{1/\alpha}.
\end{equation}
Here $Q_j$ is the spatial dilation factor. Define
\begin{equation}\label{eq:rescale}
 \tau=q_j^p(t-t_j),
 \qquad
 v_j(\theta,\tau)
 =\frac{v(\theta,t_j+\tau/q_j^p)}{q_j}.
\end{equation}
Then
\begin{equation}\label{eq:rescaled-v}
 (v_j)_\tau
 =\frac{\alpha v_j^p}{\ph(\theta)}
 \left((v_j)_{\theta\theta}+v_j-\eta_j(\tau)\right),
 \qquad
 \eta_j(\tau)=\frac{\lambda(t_j+\tau/q_j^p)}{q_j}.
\end{equation}
On the rescaled interval $[-q_j^pt_j,H_j)$, whose endpoints tend to $-\infty$ and $+\infty$, respectively, the maximality in \eqref{eq:Hj} gives
\begin{equation}\label{eq:Ham-bound}
 v_j^p(\theta,\tau)
 \le\frac{H_j}{H_j-\tau}
\end{equation}
whenever the rescaled flow is defined and $\tau<H_j$. In particular,
\[
 v_j(\theta_j,0)=1,
 \qquad
 0<v_j(\theta,0)\le1.
\]

\subsection{Space-time compactness}

\begin{lemma}\label{lem:rescaled-gradient}
For every $R>0$ there exists $C_R$ independent of $j$ such that
\begin{equation}\label{eq:rescaled-grad}
 \sup_{\Sm\times[-R,R]}
 \left(v_j+|(v_j)_\theta|\right)\le C_R
\end{equation}
for all sufficiently large $j$.
\end{lemma}

\begin{proof}
Set
\[
 C_I:=\max_{\Sm}\bigl(v^2+v_\theta^2\bigr)(\theta,0),
\]
which depends only on the initial curve. Since $(v_j)_\theta=v_\theta/q_j$, Lemma~\ref{lem:gradient} gives
\[
 (v_j)_\theta^2+v_j^2
 \le
 \max\left\{
 C_\ph\frac{\sup_{0\le s\le t}v_{max}^2(s)}{q_j^2},
 \frac{C_I}{q_j^2}
 \right\},
\]
where $t=t_j+\tau/q_j^p$. The maximality in \eqref{eq:Hj} gives, for $0\le s\le t<T_j$,
\[
 \frac{v_{max}^p(s)}{q_j^p}
 \le\frac{T_j-t_j}{T_j-s}
 \le\frac{T_j-t_j}{T_j-t}
 =\frac{H_j}{H_j-\tau}.
\]
Taking the power $2/p$ and using $H_j,q_j\to\infty$ proves \eqref{eq:rescaled-grad}.
\end{proof}

\begin{lemma}[Time equicontinuity]\label{lem:time-equicont}
For every $R>0$ there exists $C_R$ independent of $j$ such that
\begin{equation}\label{eq:time-holder}
 |v_j(\theta,\tau_2)-v_j(\theta,\tau_1)|
 \le C_R|\tau_2-\tau_1|^{1/2}
\end{equation}
for $|\tau_i|\le R$ and all sufficiently large $j$.
\end{lemma}

\begin{proof}
By \eqref{eq:lambda-upper} and \eqref{eq:profile-length}, $\eta_j\le q_j^{-1}(K_\sigma/\ell_0)^\alpha$, where $\ell_0=\inf_{t<T}L(t)>0$. Write \eqref{eq:rescaled-v} as
\[
 (v_j)_\tau=\partial_\theta F_j+G_j,
\]
where
\[
 \qquad F_j=\frac{\alpha v_j^p}{\ph}(v_j)_\theta,\quad G_j=-\left(\frac{\alpha v_j^p}{\ph}\right)_\theta(v_j)_\theta
      +\frac{\alpha v_j^p}{\ph}(v_j-\eta_j).
\]
Lemma~\ref{lem:rescaled-gradient} gives $|F_j|+|G_j|\le C_R$. Fix an arbitrary $\theta_0\in\Sm$. For $0<h\le1$, let $I_h=(\theta_0-h,\theta_0+h)$ and define
$A_{j,h}(\tau)=(2h)^{-1}\int_{I_h}v_j(\theta,\tau)\,d\theta$, using periodic extension if necessary. The spatial Lipschitz bound and integration of the equation give, respectively,
\[
 |v_j(\theta_0,\tau)-A_{j,h}(\tau)|\le C_Rh,
 \qquad
 |A_{j,h}(\tau_2)-A_{j,h}(\tau_1)|
 \le C_R\left(\frac\Delta h+\Delta\right),
\]
where $\Delta=|\tau_2-\tau_1|$. Therefore
\begin{align*}
 |v_j(\theta_0,\tau_2)-v_j(\theta_0,\tau_1)|
 &\le |v_j(\theta_0,\tau_2)-A_{j,h}(\tau_2)|\\
 &\quad+|A_{j,h}(\tau_2)-A_{j,h}(\tau_1)|\\
 &\quad+|A_{j,h}(\tau_1)-v_j(\theta_0,\tau_1)|\\
 &\le C_R\left(h+\frac\Delta h+\Delta\right).
\end{align*}
For $0<\Delta\le1$, take $h=\sqrt\Delta$. For $\Delta\ge1$, use $|v_j(\theta_0,\tau_2)-v_j(\theta_0,\tau_1)|\le2\|v_j\|_\infty\le C_R\sqrt\Delta$. This proves \eqref{eq:time-holder}.
\end{proof}

\subsection{Disappearance of the nonlocal term}

Define the rescaled Euclidean length and the rescaled numerator of the nonlocal term by
\begin{equation}\label{eq:rescaled-length}
 \begin{aligned}
 \mathcal L_j(\tau)
 &=Q_jL(t_j+\tau/q_j^p)
 =\int_{\Sm}\ph v_j^{-1/\alpha}\,d\theta,\\
 J_j(\tau)&=\int_{\Sm}\ph v_j^{1-1/\alpha}\,d\theta.
 \end{aligned}
\end{equation}
Then
\begin{equation}\label{eq:eta-length}
 \eta_j(\tau)=\frac{J_j(\tau)}{\mathcal L_j(\tau)}.
\end{equation}

\begin{lemma}\label{lem:eta}
On every compact rescaled time interval,
\begin{equation}\label{eq:eta-zero}
 \eta_j\longrightarrow0,
 \qquad
 \eta_j'\longrightarrow0
\end{equation}
uniformly as $j\to\infty$.
\end{lemma}

\begin{proof}
Since $\mathcal L_j\ge Q_j\inf_{t<T}L(t)\to\infty$, Jensen's inequality gives
\[
 0<\eta_j\le\left(\frac{K_\sigma}{\mathcal L_j}\right)^\alpha\longrightarrow0
\]
uniformly on compact rescaled time intervals. Differentiating \eqref{eq:rescaled-length}, using \eqref{eq:rescaled-v} and periodic integration by parts, we obtain
\begin{align}
 \mathcal L_j'&=-\int_{\Sm}v_j\,d\theta+2m\pi\eta_j,\label{eq:Ljprime}\\
 J_j'&=(\alpha-1)\int_{\Sm}
 \bigl(-(v_j)_\theta^2+v_j^2-\eta_jv_j\bigr)\,d\theta.\label{eq:Jjprime}
\end{align}
Indeed, the powers of $v_j$ cancel in the first differentiation, while in the second they leave the factor $v_j$:
\[
 J_j'=(\alpha-1)\int_{\Sm}v_j
 \bigl((v_j)_{\theta\theta}+v_j-\eta_j\bigr)\,d\theta.
\]
The right-hand sides of \eqref{eq:Ljprime} and \eqref{eq:Jjprime} are uniformly bounded on each compact rescaled time interval by Lemma~\ref{lem:rescaled-gradient}. The quotient rule therefore gives
\[
 |\eta_j'|
 \le\frac{|J_j'|}{\mathcal L_j}
    +\eta_j\frac{|\mathcal L_j'|}{\mathcal L_j}
 \longrightarrow0
\]
uniformly there as well.
\end{proof}

\subsection{A Harnack estimate}

Since $v_j>0$, define
\begin{equation}\label{eq:zj}
 z_j=\frac{(v_j)_\tau}{v_j}.
\end{equation}
A direct calculation using \eqref{eq:rescaled-v} gives
\begin{equation}\label{eq:z-equation}
 (z_j)_\tau
 =\frac{\alpha v_j^p}{\ph}(z_j)_{\theta\theta}
 +\frac{2\alpha v_j^{p-1}(v_j)_\theta}{\ph}(z_j)_\theta
 +pz_j^2+\frac{\alpha v_j^{p-1}}{\ph}(\eta_jz_j-\eta_j').
\end{equation}

\begin{lemma}[Harnack estimate]\label{lem:harnack}
Fix $R>0$. There exists a sequence
$\varepsilon_j(R)\to0$ such that,
for all sufficiently large $j$,
\begin{equation}\label{eq:harnack}
 z_j(\theta,\tau)
 \ge -\frac1{\tau+R}-\sqrt{\varepsilon_j(R)},
 \qquad -R<\tau\le R.
\end{equation}
Consequently, if $-R<\tau_1<\tau_2\le R$, then
\begin{equation}\label{eq:integrated-harnack}
 v_j(\theta,\tau_2)
 \ge
 \frac{R+\tau_1}{R+\tau_2}
 \exp[-c_j(R)(\tau_2-\tau_1)]
 v_j(\theta,\tau_1),
\end{equation}
where $c_j(R)\to0$.
\end{lemma}

\begin{proof}
Let $m_j(\tau)=\min_{\Sm}z_j(\cdot,\tau)$. At a spatial minimum of $z_j$, equation \eqref{eq:z-equation} and Lemmas~\ref{lem:rescaled-gradient} and \ref{lem:eta} give, in the lower Dini derivative sense,
\[
 m_j'\ge pm_j^2-C_R\|\eta_j\|_\infty|m_j|-C_R\|\eta_j'\|_\infty
 \ge \frac p2m_j^2-\varepsilon_j(R)
 \ge m_j^2-\varepsilon_j(R),
\]
where $p\ge2$ has been used, $\varepsilon_j(R)=C_R(\|\eta_j\|_\infty^2+\|\eta_j'\|_\infty)\to0$, and the norms are taken on $[-R,R]$. The function
\[
 y_j(\tau)=-\frac1{\tau+R}-\sqrt{\varepsilon_j(R)}
\]
satisfies $y_j'\le y_j^2-\varepsilon_j(R)$ and tends to $-\infty$ as $\tau\downarrow-R$. Comparison gives $m_j\ge y_j$, proving \eqref{eq:harnack}. Integrating $\partial_\tau\log v_j=z_j$ and exponentiating gives \eqref{eq:integrated-harnack}, with $c_j(R)=\sqrt{\varepsilon_j(R)}$. 
\end{proof}

\subsection{Identification of the profile}

\begin{proof}[Proof of Theorem~\ref{thm:profile}]
After passing to a subsequence,
\begin{equation}\label{eq:theta-limit}
 \theta_j\longrightarrow\theta_*\in\Sm.
\end{equation}
Set $u_j(\xi,\tau)=v_j(\theta_j+\xi,\tau)$ and $b_j(\xi)=\alpha\ph(\theta_j+\xi)^{-1}$. Then $b_j\to b_\infty:=\alpha\ph(\theta_*+\xi)^{-1}$ smoothly. Lemmas~\ref{lem:rescaled-gradient} and \ref{lem:time-equicont}, followed by Arzel\`a--Ascoli and a diagonal argument, give
\begin{equation}\label{eq:uj-limit}
 u_j\longrightarrow u_\infty
 \qquad\hbox{locally uniformly on }\Sm\times\R.
\end{equation}
By \eqref{eq:Ham-bound},
\begin{equation}\label{eq:u0}
 0\le u_\infty\le1,
 \qquad u_\infty(0,0)=1.
\end{equation}
Fix $\tau_1<\tau_2$ and $R>\max\{|\tau_1|,|\tau_2|\}+1$. Passing to the limit in \eqref{eq:integrated-harnack}, with $R$ fixed, yields
\[
 u_\infty(\xi,\tau_2)\ge
 \frac{R+\tau_1}{R+\tau_2}u_\infty(\xi,\tau_1).
\]
Letting $R\to\infty$ gives
\begin{equation}\label{eq:time-monotone}
 u_\infty(\xi,\tau_2)\ge u_\infty(\xi,\tau_1),
 \qquad\tau_2>\tau_1.
\end{equation}
In particular, \eqref{eq:u0} implies
\begin{equation}\label{eq:u-origin-one}
 u_\infty(0,\tau)=1,\qquad\tau\ge0.
\end{equation}

Let $I$ be the connected component of $\{u_\infty(\cdot,0)>0\}$ containing $0$. By \eqref{eq:time-monotone} and local uniform convergence, $u_j$ has a uniform positive lower bound on each $K\times[\delta,R]$, where $K\Subset I$ and $0<\delta<R$. Standard interior parabolic estimates and Lemma~\ref{lem:eta} yield
\begin{equation}\label{eq:smooth-positive}
 u_j\longrightarrow u_\infty
 \quad\hbox{in }C^{2,1}_{\mathrm{loc}}(I\times(0,\infty)).
\end{equation}
The limit satisfies
\begin{equation}\label{eq:limit-PDE}
 (u_\infty)_\tau=b_\infty(\xi)u_\infty^p
 \bigl((u_\infty)_{\xi\xi}+u_\infty\bigr),
 \qquad b_\infty(\xi)=\frac\alpha{\ph(\theta_*+\xi)},
\end{equation}
and is smooth there by interior regularity.

By \eqref{eq:time-monotone}, we have
$w:=(u_\infty)_\tau\ge0$. Differentiating \eqref{eq:limit-PDE} gives
\begin{equation}\label{eq:w-equation}
 w_\tau=b_\infty u_\infty^pw_{\xi\xi}+cw,
 \qquad
 c=pb_\infty u_\infty^{p-1}(u_\infty)_{\xi\xi}+(p+1)b_\infty u_\infty^p.
\end{equation}
This equation is uniformly parabolic on compact subsets of $I\times(0,\infty)$, with locally bounded coefficients. Since $w\ge0$ and $w(0,\tau)=0$ for every $\tau>0$, the strong maximum principle gives $w\equiv0$ on $I\times(0,\infty)$.

It follows that $u_\infty$ is independent of $\tau$ and
$(u_\infty)_{\xi\xi}+u_\infty=0$ on $I$. Since $u_\infty(0)=1$ and $(u_\infty)_\xi(0)=0$, we obtain $u_\infty(\xi,\tau)=\cos\xi$. By continuity the identity also holds at $\tau=0$, and maximality of the positive component gives $I=(-\pi/2,\pi/2)$. Hence \eqref{eq:uj-limit} gives
\begin{equation}\label{eq:moving-profile-proof}
 \frac{v(\theta_j+\xi,t_j)}{v(\theta_j,t_j)}\longrightarrow\cos\xi
 \quad\hbox{uniformly for }|\xi|\le\pi/2.
\end{equation}
Finally, the gradient estimate gives
\[
 \sup_{|\xi|\le\pi/2}
 \left|\frac{v(\theta_*+\xi,t_j)-v(\theta_j+\xi,t_j)}{q_j}\right|
 \le C|\theta_j-\theta_*|\longrightarrow0.
\]
Taking $\xi=0$ shows that $v(\theta_*,t_j)/q_j\to1$. Therefore
\[
 \frac{v(\theta_*+\xi,t_j)}{v(\theta_*,t_j)}\longrightarrow\cos\xi
\]
uniformly on the same interval. Taking the power $1/\alpha$ and using $v=(\ph\kappa)^\alpha$ proves \eqref{eq:profile-moving} and \eqref{eq:profile-fixed}, and completes the proof.
\end{proof}

\begin{remark}\label{rem:profile-comparison}
Theorems~\ref{thm:typeII} and \ref{thm:profile} give the Type-II conclusion without the symmetry assumption in the isotropic work of Anada--Ishiwata--Ushijima \cite{AnadaIshiwataUshijima2025}, and also allow anisotropy, under the stated length lower bound. For the anisotropic flow of Poon--Tsai \cite[Theorem~4.7 and Corollary~4.9]{PoonTsai2015}, the nonlocal error is absent. Their geometric Type-II result is stated for embedded convex curves, whereas the same profile argument here applies to immersed locally convex curves without symmetry. These conclusions concern selected blow-up subsequences and do not assert uniqueness of the blow-up direction.
\end{remark}

\section{Global compactness and finite-time singularity}\label{sec:global}

\subsection{Global compactness and proof of Theorem~\ref{thm:finite}}\label{subsec:finite-proof}

We now turn to the sufficient conditions for finite-time singularity formation. The following convergence result will be used to prove Theorem~\ref{thm:finite}.

\begin{proposition}[Global compactness and convergence]\label{prop:global}
Assume that a smooth locally convex solution of \eqref{eq:main-flow} exists for all $t\ge0$ and that
\begin{equation}\label{eq:Ls-positive-limit}
 \Ls(\infty):=\lim_{t\to\infty}\Ls(t)>0.
\end{equation}
Then $v=\ks^\alpha$ has positive time-independent upper and lower bounds, and all its spatial derivatives are uniformly bounded. Moreover,
\begin{equation}\label{eq:diss-to-zero}
 \int_{X_t}(\ks-\lambda^{1/\alpha})(\ks^\alpha-\lambda)\,ds\longrightarrow0
 \qquad\hbox{as }t\to\infty,
\end{equation}
and, for every integer $k\ge0$,
\[
 \lambda(t)\longrightarrow C_*^\alpha,\qquad
 v(\cdot,t)\longrightarrow C_*^\alpha\quad\hbox{in }C^k(\Sm),
 \qquad C_*:=\frac{2|\Ws|}{\Ls(\infty)}>0.
\]
After translating each curve so that $X(0,t)=0$, the curves converge smoothly to a translation of $C_*^{-1}\partial\Ws$. In particular,
\begin{equation}\label{eq:C-area}
 A(0)=\frac{|\Ws|}{C_*^2}>0,
 \qquad C_*=\sqrt{\frac{|\Ws|}{A(0)}}.
\end{equation}
\end{proposition}

\begin{proof}
By \eqref{eq:Ls-positive-limit} and positivity of $\sigma$,
\begin{equation}\label{eq:L-two-sided}
 0<\frac{\Ls(\infty)}{\sigma_{max}}
 \le L(t)\le\frac{\Ls(0)}{\sigma_{min}}<\infty.
\end{equation}
Thus \eqref{eq:lambda-upper} gives a positive time-independent upper bound for $\lambda$. Write $\lambda\le\Lambda$ and set
\[
 \mathcal D(t):=-\Ls'(t)
 =\int_{X_t}(v^{1/\alpha}-\lambda^{1/\alpha})(v-\lambda)\,ds.
\]
Lemma~\ref{lem:monotone} gives $\mathcal D\ge0$ and
$\int_0^\infty\mathcal D(t)\,dt=\Ls(0)-\Ls(\infty)<\infty$.
For any fixed $d_0>0$, there are $t_j\in[(j-1)d_0,jd_0]$ with $\mathcal D(t_j)\to0$. By the first variation of anisotropic length,
\begin{equation}\label{eq:D-F}
 \mathcal D(t)=\int_{\Sm}\ph v\,d\theta-\lambda K_\sigma.
\end{equation}
Consequently, for $F(t):=\int_{\Sm}v\,d\theta$,
\[
 F(t_j)\le B:=\frac{1+\Lambda K_\sigma}{\ph_{min}}
\]
for all sufficiently large $j$, where $B$ is independent of $d_0$.

To extend the bound to all times, put $a=\alpha\ph^{-1}$. By \eqref{eq:v-main}, periodic integration by parts and $\lambda\ge0$ give
\[
 F'(t)\le-p\int_{\Sm}av^{p-1}v_\theta^2\,d\theta
 +\int_{\Sm}\left(a+\frac1{p+1}a_{\theta\theta}\right)v^{p+1}\,d\theta.
\]
Since $v_s=(v^{1/\alpha}/\ph)v_\theta$, $d\theta=(v^{1/\alpha}/\ph)\,ds$ and $p-1=1/\alpha$, we obtain
\begin{equation}\label{eq:Fprime-first}
 F'(t)\le-\alpha p\int_{X_t}v_s^2\,ds+C\int_{X_t}v^{2p}\,ds.
\end{equation}
The one-dimensional Gagliardo--Nirenberg inequality gives
\[
 \|v\|_{L^{2p}}\le C\left(
 \|v_s\|_{L^2}^{1/(p+2)}\|v\|_{L^p}^{(p+1)/(p+2)}
 +\|v\|_{L^p}\right),
\]
where all norms are taken along $X_t$ and the constant is uniform by \eqref{eq:L-two-sided}. Taking the power $2p$ and applying Young's inequality yields
\[
 \int_{X_t}v^{2p}\,ds\le\varepsilon\int_{X_t}v_s^2\,ds
 +C_\varepsilon\left(\|v\|_{L^p}^{p(p+1)}+\|v\|_{L^p}^{2p}\right),
\]
where $C_\varepsilon>0$ is a time-independent constant. Using
$\|v\|_{L^p}^p=\int_{\Sm}\ph v\,d\theta\le\ph_{max}F$ and taking $\varepsilon$ sufficiently small to absorb the derivative term in \eqref{eq:Fprime-first}, we obtain
\begin{equation}\label{eq:F-ode}
 F'(t)\le C\bigl(F^{p+1}(t)+F^2(t)\bigr).
\end{equation}
Thus there is a number $\delta_0>0$, independent of $j$, such that a fixed bound for $F(t_j)$ implies a fixed bound for $F$ on $[t_j,t_j+\delta_0]$. Choosing $d_0=\delta_0/2$ makes these intervals overlap, and hence
\begin{equation}\label{eq:F-global}
 \sup_{t\ge0}F(t)<\infty.
\end{equation}
If $v$ were not uniformly bounded, choose $(\theta_j,s_j)$ such that
\[
 v(\theta_j,s_j)=\max_{\Sm\times[0,s_j]}v\longrightarrow\infty.
\]
Lemma~\ref{lem:concentration} gives a fixed-length angle interval on which $v(\cdot,s_j)$ is comparable with this maximum. Thus $F(s_j)\to\infty$, contradicting \eqref{eq:F-global}. Hence
\begin{equation}\label{eq:v-upper-global}
 v(\theta,t)\le C, \quad t\in[0,\infty),
\end{equation}
and Lemma~\ref{lem:gradient} also gives $|v_\theta|\le C$.
If $e_j:=v(\theta_j,t_j)\to0$, this Lipschitz bound would imply
\[
 \Ls(t_j)\ge c\int_0^\delta\frac{dr}{(e_j+Cr)^{1/\alpha}}
 \longrightarrow\infty,
\]
since $1/\alpha\ge1$. Here $c,\delta>0$ are independent of $j$. This contradicts $\Ls(t_j)\le\Ls(0)$. Therefore
\begin{equation}\label{eq:v-lower-global}
 0<c\le v(\theta,t)\le C.
\end{equation}
The same bounds hold for $\lambda$, since it is the arc-length average of $v$.

The equation is now uniformly parabolic. Differentiating $J$ in \eqref{eq:lambda-main} and integrating by parts, as in \eqref{eq:Jjprime}, gives
\[
 J'=(\alpha-1)\int_{\Sm}(-v_\theta^2+v^2-\lambda v)\,d\theta.
\]
This formula, \eqref{eq:Lprime-basic} and \eqref{eq:L-two-sided} bound $\lambda'=J'/L-\lambda L'/L$. Standard interior estimates then give time-independent bounds for $v_t$ and all spatial derivatives of $v$. Differentiating \eqref{eq:D-F} therefore gives $|\mathcal D'(t)|\le C$.
Since $\mathcal D$ is nonnegative and integrable, this bound implies $\mathcal D(t)\to0$. Moreover,
\[
 \mathcal D(t)
 =\int_{\Sm}\frac{\ph}{v^{1/\alpha}}
 (v^{1/\alpha}-\lambda^{1/\alpha})(v-\lambda)\,d\theta
 \ge c\|v(\cdot,t)-\lambda(t)\|_{L^2(\Sm)}^2.
\]
The last inequality follows from \eqref{eq:v-lower-global} and the mean value theorem applied to $s\mapsto s^{1/\alpha}$ on the common positive range of $v$ and $\lambda$.
The uniform spatial derivative estimates and compactness therefore give
\begin{equation}\label{eq:v-minus-lambda}
 \|v(\cdot,t)-\lambda(t)\|_{C^k(\Sm)}\longrightarrow0
 \qquad\hbox{for every integer }k\ge0.
\end{equation}

To identify the limit, use
\[
 \lambda(t)^{1/\alpha}\Ls(t)-2|\Ws|
 =\int_{\Sm}\frac{\sigma\ph}{v^{1/\alpha}}
 (\lambda^{1/\alpha}-v^{1/\alpha})\,d\theta.
\]
Thus $\bigl|\lambda^{1/\alpha}\Ls-2|\Ws|\bigr|
\le\Ls\|v^{1/\alpha}-\lambda^{1/\alpha}\|_\infty\to0$, and
\[
 \lambda(t)^{1/\alpha}\longrightarrow C_*:=\frac{2|\Ws|}{\Ls(\infty)},
 \qquad v(\cdot,t)\longrightarrow C_*^\alpha\quad\hbox{in every }C^k(\Sm).
\]
Finally, the translated curves satisfy
\[
 X(\theta,t)-X(0,t)
 =\int_0^\theta\frac{\ph(r)}{v(r,t)^{1/\alpha}}T(r)\,dr
 \longrightarrow\frac1{C_*}\int_0^\theta\ph(r)T(r)\,dr
\]
smoothly. Since $\partial_\theta(-\sigma N+\sigma'T)=\ph T$, the limit is a translation of $C_*^{-1}\partial\Ws$. Continuity and translation invariance of signed area, together with area preservation, give
$A(0)=|\Ws|/C_*^2>0$, proving \eqref{eq:C-area}.
\end{proof}

\begin{proof}[Proof of Theorem~\ref{thm:finite}]
Suppose first that $A(0)<0$ and the solution were global. A generalized isoperimetric inequality for immersed closed plane curves (see Osserman \cite{Osserman1978}) gives
\[
 L^2(t)\ge4\pi|A(t)|=4\pi|A(0)|.
\]
Therefore
\[
 \Ls(t)\ge\sigma_{min}L(t)
 \ge2\sigma_{min}\sqrt{\pi|A(0)|},
\]
so \eqref{eq:Ls-positive-limit} holds. Proposition~\ref{prop:global} would then imply that the preserved signed area is positive, a contradiction.

Now assume
\[
 0<\Ls^2(0)<4|\Ws|A(0).
\]
Again $A(0)>0$ and the generalized isoperimetric inequality gives \eqref{eq:Ls-positive-limit}. If the flow were global, Proposition~\ref{prop:global} would give
\[
 \ks(\cdot,t)\longrightarrow
 C_*=\sqrt{\frac{|\Ws|}{A(0)}}.
\]
Consequently,
\[
 \Ls(\infty)
 =\frac{1}{C_*}\Ls(\partial\Ws)
 =\frac{2|\Ws|}{C_*}
 =2\sqrt{|\Ws|A(0)}.
\]
Since $\Ls$ is non-increasing,
\[
 \Ls^2(0)\ge\Ls^2(\infty)=4|\Ws|A(0),
\]
contradicting the assumption. Hence in either case the maximal existence time is finite. By the continuation criterion, the curvature blows up there.
\end{proof}

\subsection{Initial curves near the Wulff curve}\label{subsec:wulff-examples}

We give the construction asserted in Remark~\ref{rem:wulff-perturbation}.

For a locally convex immersion parametrized by its tangential angle, its support function relative to the origin is $h=-\langle X,N\rangle$. Then $h_\theta=\langle X,T\rangle$, so
\[
 X=-hN+h_\theta T,\qquad
 X_\theta=(h+h_{\theta\theta})T.
\]
Conversely, any smooth $2m\pi$-periodic function $h$ with $h+h_{\theta\theta}>0$ determines in this way a closed locally convex immersion of rotation index $m$. Here the support function is understood along the immersion, not as the support function of a convex body.

Put $B=|\Ws|$ and define
\[
 c_0:=\frac{\displaystyle\int_{\Sm}\ph(\theta)\cos(\theta/m)\,d\theta}
 {\displaystyle\int_{\Sm}\ph(\theta)\,d\theta},\qquad
 f(\theta):=\cos(\theta/m)-c_0.
\]
Then $\int_{\Sm}\ph f\,d\theta=0$.
For sufficiently small $|\varepsilon|$, the function $h_\varepsilon=\sigma+\varepsilon f$ satisfies $h_\varepsilon+h_{\varepsilon,\theta\theta}>0$, and hence defines a locally convex closed curve
$X_\varepsilon=-h_\varepsilon N+(h_\varepsilon)_\theta T$.
Periodic integration by parts gives
\begin{align*}
 \Ls(X_\varepsilon)
 &=\int_{\Sm}\sigma(h_\varepsilon+h_{\varepsilon,\theta\theta})\,d\theta=2B,\\
 A(X_\varepsilon)
 &=\frac12\int_{\Sm}h_\varepsilon
 (h_\varepsilon+h_{\varepsilon,\theta\theta})\,d\theta
 =B+\frac{\varepsilon^2}{2}\int_{\Sm}(f^2-f_\theta^2)\,d\theta.
\end{align*}
Since
\[
 \int_{\Sm}(f^2-f_\theta^2)\,d\theta
 =m\pi\left(1-\frac1{m^2}\right)+2m\pi c_0^2>0,
\]
we obtain
\[
 0<\Ls^2(X_\varepsilon)=4B^2<4B A(X_\varepsilon)
\]
for all sufficiently small nonzero $\varepsilon$.
Thus, for every $0<\alpha\le1$, Corollary~\ref{cor:finite-typeII} gives a finite-time Type-II singularity for each such initial curve. Moreover, $X_\varepsilon\to-\sigma N+\sigma_\theta T$ in $C^\infty(\Sm)$ as $\varepsilon\to0$. The strict inequality persists under sufficiently small smooth perturbations of each $X_\varepsilon$, so the examples need not be confined to a symmetric class.

\section*{Acknowledgements}
This work is partially supported by the NSF of China 12471193, 12271092, 11871148, 11101078, the 2023 Qinglan Project of Jiangsu Province, and the Jiangsu Provincial Scientific Research Center of Applied Mathematics under Grant No.~BK20233002. 

\section*{Data availability statement}
Data sharing is not applicable to this article as
no datasets were generated or analysed during
the current study.

\section*{Conflict of interest statement}
The authors state no conflict of interest.

\end{document}